\documentclass[11pt]{amsart}
\usepackage{amssymb}
\usepackage[colorlinks]{hyperref}

\newtheorem{theorem}{Theorem}[section]
\newtheorem{lem}[theorem]{Lemma}

\newtheorem{cor}[theorem]{Corollary}

\theoremstyle{definition}

\theoremstyle{remark}

\numberwithin{equation}{section}

\begin{document}

\newcommand{\spacing}[1]{\renewcommand{\baselinestretch}{#1}\large\normalsize}
\spacing{1.14}

\title{Riemannian Geometry of Two Families of Tangent Lie Groups }

\author {F. Asgari}

\address{Department of Mathematics\\ Faculty of  Sciences\\ University of Isfahan\\ Isfahan\\ 81746-73441-Iran.} \email{farhad\_13812003@yahoo.com}

\author {H. R. Salimi Moghaddam}

\address{Department of Mathematics\\ Faculty of  Sciences\\ University of Isfahan\\ Isfahan\\ 81746-73441-Iran.} \email{hr.salimi@sci.ui.ac.ir and salimi.moghaddam@gmail.com}


%%\date{\today}

\begin{abstract}
Using vertical and complete lifts, any left invariant Riemannian metric on a Lie group induces a left invariant Riemannian metric on the tangent Lie group. In the present article we study the Riemannian geometry of tangent bundle of two families of Lie groups. The first one is the family of special Lie groups considered by J. Milnor and the second one is the class of Lie groups with one-dimensional commutator groups. The Levi-Civita connection, sectional and Ricci curvatures have been investigated.
\end{abstract}

\maketitle

%%---------------------------INTRODUCTION--------------------------

\section{\textbf{Introduction}}\label{introduction}
Suppose that $M$ is a real $m-$dimensional differentiable manifold. For any vector field $X$ on $M$, the infinitesimal generator of the one-parameter group of diffeomorphisms $\psi_t(y):=y+tX(x), \forall y\in T_xM$, is a vector field on $TM$ which is called the vertical lift of $X$ and is denoted by $X^v$. Let $\phi_t$ be the (local) one-parameter group of diffeomophisms defined by $X$ on $M$. The infinitesimal generator of $T\phi_t:TM\longrightarrow TM$, is called the complete lift of $X$ and is denoted by $X^c$ (see \cite{Hind, MoFeLoMaRu, YaIs}). For any two vector fields $X, Y$ on $M$ the Lie bracket of vertical and complete lifts of them satisfy the following relations (see \cite{MoFeLoMaRu, YaIs}):
\begin{equation}\label{Lie bracket}
    [X^v,Y^v]=0  \ \ \ \ , \ \ \ \ [X^c,Y^c]=[X,Y]^c \ \ \ \ , \ \ \ \ [X^v,Y^c]=[X,Y]^v.
\end{equation}
Let $G$ be a Lie group with multiplication map $\mu$ and inversion map $\iota$. Then $TG$ is also a Lie group with multiplication $T\mu$ and inversion map $T\iota$, where $T\mu$ and $T\iota$ denote the tangent maps of $\mu$ and $\iota$ respectively.\\
More precisely, $TG$ is a Lie group with the multiplication
\begin{equation}\label{multiplication of TG}
    T\mu(v,w)=T_hl_gw+T_gr_hv, \ \ \ \ \forall g,h\in G, v\in T_gG, w\in T_hG,
\end{equation}
where $l_g$ and $r_h$ denote the left and right translations respectively.\\
In \cite{Hind}, it is shown that the complete and vertical lifts of left invariant vector fields of $G$ are left invariant vector fields of $TG$. Therefore for any left invariant Riemannian metric $g$ on $G$ we can define a left invariant Riemannian metric $\widetilde{g}$ on $TG$ as follows:
\begin{eqnarray}
  \widetilde{g}(X^c,Y^c) &:=& g(X,Y)=:\widetilde{g}(X^v,Y^v), \\
  \widetilde{g}(X^c,Y^v) &=& 0, \nonumber
\end{eqnarray}
where $X, Y$ are arbitrary left invariant vector fields on $G$.\\
In \cite{Mi}, Milnor considered a special class of solvable Lie groups (we denote this family by $\mathcal{G}_1$). By definition a non-abelian Lie group $G$ belongs to $\mathcal{G}_1$ (or it is called a special Lie group) if for any $x, y\in \frak{g}$, $[x,y]$ is a linear combination of $x$ and $y$, where $\frak{g}$ denotes the Lie algebra of $G$. Milnor showed that if $\frak{g}$ has this property then there exist a commutative ideal $\frak{u}$ of codimension $1$ and an element $b\in\frak{g}$ such that $[b,x]=x$ for any $x\in\frak{u}$.\\
In \cite{Mi} it has been shown that if $G\in\mathcal{G}_1$, then every left-invariant Riemannian metric on $G$ has negative constant sectional curvature. This result generalized to left-invariant Lorentz metric by K. Nomizu (see \cite{Nomizu}). He proved that every left-invariant Lorentz metric on a special Lie group $G$ is of constant sectional curvature depending on choice of left-invariant Lorentz metric may be positive, negative, or zero.\\
Another interesting family of Lie groups is the class of Lie groups with one-dimensional commutator groups (we denote it by $\mathcal{G}_2$). The Ricci curvatures of such Lie groups, equipped with left invariant Riemannian metrics, have been studied by V. Kaiser (see \cite{Kaiser}).\\
In the previous work \cite{Asgri-Salimi}, we have studied the Riemannian geometry of the tangent Lie group $(TG,\widetilde{g})$. In this paper we study the sectional and Ricci curvatures of the manifolds $(TG,\widetilde{g})$ where $G$ belongs to the families $\mathcal{G}_1$ and $\mathcal{G}_2$ equipped with any left invariant Riemannian metric $g$.

%%---------------------------RSectional and Ricci Curvatures of $TG$ Where $G$ Belongs to $\mathcal{G}_1$--------------------------
\section{\textbf{Sectional and Ricci Curvatures of $TG$ Where $G$ Belongs to $\mathcal{G}_1$}}
Suppose that $(G,g)$ is an arbitrary finite dimensional Riemannian Lie group. In \cite{Asgri-Salimi}, we showed that for any $x,y\in\frak{g}$ the Levi-Civita connection of the lifted left invariant metric $\widetilde{g}$ on $TG$ can be computed by the following equations:
\begin{eqnarray}\label{tangent bundle connection}
&&\widetilde{\nabla}_{x^c}y^c=(\nabla_xy)^c,\nonumber \\
&&\widetilde{\nabla}_{x^v}y^v=(\nabla_xy-\frac{1}{2}[x,y])^c,\nonumber \\
&&\widetilde{\nabla}_{x^c}y^v=(\nabla_xy+\frac{1}{2}\textsl{ad}_y^\ast x)^v,\\
&&\widetilde{\nabla}_{x^v}y^c=(\nabla_xy+\frac{1}{2}\textsl{ad}_x^\ast y)^v.\nonumber
\end{eqnarray}
Let $G\in\mathcal{G}_1$ be a Lie group equipped with an arbitrary left invariant Riemannian metric $g$. Suppose that  $\frak{u}$ is its commutative ideal of codimension $1$. Similar to \cite{Nomizu}, it can be shown that there exists a vector $b\in\frak{g}$ orthogonal to $\frak{u}$ (with respect to the inner product induced by $g$ on $\frak{g}$) such that $[b,x]=x$ for any $x\in\frak{u}$. Then we can see for every $x,y \in \frak{u}$ we have:\\
\begin{eqnarray}
\nabla_bb=0, \ \ \ \ \nabla_bx=0, \ \ \ \ \nabla_xy=\frac{g(x,y)}{\lambda}b, \ \ \ \ \nabla_xb=-x,
\end{eqnarray}
where $ \lambda=g(b,b)$.\\
Now for the Levi-Civita connection of $(TG,\widetilde{g})$ we have the following lemma:

\begin{lem}\label{lemma 1}
Let $G\in\mathcal{G}_1$ be a special Lie group, $\frak{g}$ be its Lie algebra and $g$ be any left-invariant Riemannian metric on $G$. Suppose that $\widetilde{g}$ is the left-invariant Riemannian metric on $TG$ induced by $g$ as above, and $\widetilde{\nabla}$ is the Levi-Civita connection of $(TG,\widetilde{g})$. Then for any $x,y\in\frak{u}$ we have:
\begin{eqnarray}
  &&\widetilde{\nabla}_{b^c}b^c=\widetilde{\nabla}_{b^c}b^v=\widetilde{\nabla}_{b^v}b^c=\widetilde{\nabla}_{b^v}b^v=\widetilde{\nabla}_{b^c}x^c=\widetilde{\nabla}_{b^c}x^v=0, \nonumber \\
  &&\widetilde{\nabla}_{x^c}b^c=-x^c, \ \ \ \ \widetilde{\nabla}_{b^v}x^c=\frac{1}{2}x^v, \ \ \ \ \widetilde{\nabla}_{x^v}b^c=-x^v, \ \ \ \ \widetilde{\nabla}_{x^v}b^v=\widetilde{\nabla}_{b^v}x^v=-\frac{1}{2}x^c, \\
  &&\widetilde{\nabla}_{x^c}b^v=-\frac{1}{2}x^v, \ \ \ \ \widetilde{\nabla}_{x^c}y^c=\widetilde{\nabla}_{x^v}y^v=\frac{g(x,y)}{\lambda}b^c, \ \ \ \ \widetilde{\nabla}_{x^c}y^v=\widetilde{\nabla}_{x^v}y^c=\frac{g(x,y)}{2\lambda}b^v.\nonumber
\end{eqnarray}

where $\lambda:=g(b,b)=\|b\|^2$ and $b\in\frak{g}$ is an element such that $[b,x]=x$ for any $x\in\frak{u}$.
\end{lem}
\begin{proof}

Direct computations show that for every $x\in\frak{u}$ we have:
\begin{eqnarray}
ad^*_bb=0, \ \ \ \ ad^*_bx=x, \ \ \ \ ad^*_xb=0.
\end{eqnarray}
The substitution of the above formulas in \ref{tangent bundle connection} completes the proof.
\end{proof}

The curvature tensor of a Riemannian manifold by definition is
\begin{eqnarray}
R(x,y)z=\nabla_x\nabla_yz-\nabla_y\nabla_xz-\nabla_{[x,y]}z.
\end{eqnarray}
The following lemma is a direct consequence of \ref{lemma 1} so the proof is omitted.

\begin{lem}\label{lemma 2}
If $\widetilde{R}$ denotes the curvature tensor of the Riemannian metric $\widetilde{g}$ then, with the assumptions of lemma \ref{lemma 1}, for any $x,y,z\in\frak{u}$ we have:
\begin{eqnarray}\label{curvature tensor}
   && \widetilde{R}(x^c,y^v)z^c=\widetilde{R}(x^c,y^c)z^v=\frac{1}{4\lambda}(g(x,z)y^v-g(y,z)x^v),\nonumber \\
   && \widetilde{R}(x^c,y^c)z^c=\frac{1}{\lambda}(g(x,z)y^c-g(y,z)x^c) \ \ \ , \ \ \widetilde{R}(x^c,y^v)z^v=\frac{1}{4\lambda}(g(x,z)y^c-4g(y,z)x^c),\nonumber \\
   && \widetilde{R}(x^v,y^v)z^c=\frac{1}{4\lambda}(g(x,z)y^c-g(y,z)x^c) \ \ \ , \ \ \widetilde{R}(x^v,y^v)z^v=\frac{1}{\lambda}(g(x,z)y^v-g(y,z)x^v), \\
   && \widetilde{R}(x^c,y^c)b^c=\widetilde{R}(x^c,y^c)b^v=\widetilde{R}(x^c,y^v)b^c=\widetilde{R}(x^c,y^v)b^v=\widetilde{R}(x^v,y^v)b^c=\widetilde{R}(x^v,y^v)b^v=0,\nonumber \\
   && \frac{4}{3}\widetilde{R}(x^c,b^v)y^c=2\widetilde{R}(x^c,b^c)y^v=2\widetilde{R}(x^v,b^c)y^c=-4\widetilde{R}(x^v,b^v)y^v=\frac{g(x,y)}{\lambda}b^v,\nonumber \\
   && \widetilde{R}(x^c,b^c)y^c=2\widetilde{R}(x^c,b^v)y^v=2\widetilde{R}(x^v,b^v)y^c=\widetilde{R}(x^v,b^c)y^v=\frac{g(x,y)}{\lambda}b^c.\nonumber
\end{eqnarray}
\end{lem}
Now by using the above lemma we can prove the following theorem about sectional curvature.
\begin{theorem}\label{theorem 1}
Consider the assumptions of lemma \ref{lemma 1} and let $\widetilde{K}$ denote the sectional curvature of the Riemannian metric $\widetilde{g}$, then we have:
\begin{eqnarray}
    &&\widetilde{K}(x^c,y^c)=\widetilde{K}(x^v,y^v)=\widetilde{K}(x^c,b^c)=\frac{4}{3}\widetilde{K}(x^c,b^v)=\widetilde{K}(x^v,b^c)=-4\widetilde{K}(x^v,b^v)=-\frac{1}{\lambda},\nonumber \\
    &&\widetilde{K}(x^c,y^v)=-\frac{1}{\lambda}+\frac{g(x,y)^2}{4\lambda g(x,x)g(y,y)}.
\end{eqnarray}
\end{theorem}
\begin{proof}
It is a direct consequence of substituting of the relations \ref{curvature tensor} in the sectional curvature formula.
\end{proof}
\begin{cor}
Although the special Lie group $(G,g)$ has constant negative sectional curvature but $(TG,\widetilde{g})$ obtains positive, negative and zero sectional curvatures.
\end{cor}
Let $\{u_1,\cdots,u_n\}$ be an orthonormal basis for the Lie algebra $\frak{g}$, with respect to the inner product induced on $\frak{g}$ by $g$. Then the Ricci tensor is defined by
\begin{eqnarray}
Ric(x,y)=\sum_{i=1}^n g(R(u_i,x)y,u_i).
\end{eqnarray}

The following theorem gives us the Ricci tensor of $(TG,\widetilde{g})$.

\begin{theorem}
Let $G$ be a $(n+1)-$dimensional special Lie group equipped with any left invariant Riemannian metric $g$. Suppose that $\frak{g}$ denotes the Lie algebra of $G$ and, $\frak{u}$ and $b$ are as lemma \ref{lemma 1} such that $\{u_1,\cdots, u_n, b\}$ is an orthonormal basis of $\frak{g}$. Then for Ricci tensor of the manifold $(TG,\widetilde{g})$ we have:
\begin{eqnarray}
  && \widetilde{Ric}(x^c,y^c)=-(\frac{7}{4}+\frac{2n}{\lambda})g(x,y)+\frac{5}{4\lambda}\sum_{i=1}^ng(x,u_i)g(u_i,y),\nonumber \\
  && \widetilde{Ric}(x^v,y^v)=-(\frac{3}{4}+\frac{2n}{\lambda})g(x,y)+\frac{5}{4\lambda}\sum_{i=1}^ng(x,u_i)g(u_i,y), \\
  && \widetilde{Ric}(x^c,b^c)=\widetilde{Ric}(x^c,b^v)=\widetilde{Ric}(x^v,b^v)=\widetilde{Ric}(x^v,b^c)=\widetilde{Ric}(x^c,y^v)=0.\nonumber
\end{eqnarray}
\end{theorem}
\begin{proof}
It is easy to apply lemma \ref{lemma 1} in Ricci tensor formula.
\end{proof}

%%---------------------------RSectional and Ricci Curvatures of $TG$ Where $G$ Belongs to $\mathcal{G}_2$--------------------------
\section{\textbf{Sectional and Ricci Curvatures of $TG$ Where $G$ Belongs to $\mathcal{G}_2$}}
In this section we consider the family of Lie groups with one-dimensional commutator group, $\mathcal{G}_2$. Let $G\in\mathcal{G}_2$ be a Lie group equipped with any left invariant Riemannian metric $g$, and $G'$ be its one-dimensional commutator group. Suppose that $\frak{g}$ and $\frak{g}'$ are the Lie algebra and the one-dimensional commutator algebra of $G$ and $G'$, respectively, and assume that $\mathrm{e}$ is any unit vector belongs to the ideal $\frak{g}'$ with respect to the inner product induced by $g$ on $\frak{g}$. Let $\Gamma$ be the orthogonal subspace to the vector $\mathrm{e}$ with respect to $g$ and $\phi:\Gamma\longrightarrow\mathbb{R}$ be the linear function defined by $[x,\rm{e}]=\phi(x)\rm{e}$, for any $x\in\Gamma$. Then there exists a unique element $a\in\Gamma$ such that $\phi(x)=g(a,x)$, for each $x\in\Gamma$. So we have
\begin{equation}\label{Lie bracket a}
    [x,\rm{e}]=g(a,x)\rm{e}, \ \ \ \ \ \ \forall x\in\Gamma.
\end{equation}
On the other hand there exists a skew-symmetric bilinear form $B$ on $\Gamma$ such that $[x,y]=B(x,y)\rm{e}$ for any $x,y\in\Gamma$. Therefore there is a unique skew-adjoint linear transformation $f:\Gamma\longrightarrow\Gamma$ such that $B(x,y)=g(f(x),y)$ for each $x,y\in\Gamma$. So we have:
\begin{equation}\label{Lie bracket f}
    [x,y]=g(f(x),y)\rm{e}, \ \ \ \ \ \ \forall x,y\in\Gamma.
\end{equation}

\begin{lem}\label{lemma 3}
First of all we derive the sectional curvature of based Lie group.
For the sectional curvature of $(G,g)$ we have:
\begin{eqnarray}
  K(x,y) &=& -\frac{3}{4}g(f(x),y)^2,\nonumber  \\
  K(x,\rm{e}) &=& \frac{1}{4}\|f(x)\|^2-g(a,x)^2,
\end{eqnarray}
where $\{x,y\}$ is an orthonormal set and $\|f(x)\|^2:=g(f(x),f(x))$.
\end{lem}
\begin{proof}
With the above notations for any $x,y\in\Gamma$ the Levi-Civita connection of the Riemannian manifold $(G,g)$ satisfies the following equations (see \cite{Kaiser}).
\begin{eqnarray}
  &&\nabla_{\rm{e}}\rm{e}=a, \ \ \ \ \ \ \ \ \ \ \ \ \ \ \ \nabla_{\rm{e}}x=-\frac{1}{2}f(x)-g(x,a)\rm{e},\nonumber  \\
  &&\nabla_x\rm{e}=-\frac{1}{2}f(x), \ \ \ \ \ \ \ \nabla_xy=\frac{1}{2}g(f(x),y)\rm{e}.
\end{eqnarray}
Now it is sufficient to use sectional curvature formula.
\end{proof}

Similar to the previous section in this step we compute the Levi-Civita connection of tangent Lie group.

\begin{lem}\label{lemma 4}
Let $G\in\mathcal{G}_2$, $\frak{g}$ be its Lie algebra and $g$ be any left invariant Riemannian metric on $G$. Suppose that $\widetilde{g}$ is the lifted Riemannian metric on $TG$ as considered in the section \ref{introduction} and $\widetilde{\nabla}$ is the Levi-Civita connection of $(TG,\widetilde{g})$. Then for any $x,y\in\Gamma$ we have:
\begin{eqnarray}
  && \widetilde{\nabla}_{\rm{e}^c}\rm{e}^c=\widetilde{\nabla}_{\rm{e}^v}\rm{e}^v=a^c, \ \ \ \ \ \widetilde{\nabla}_{\rm{e}^c}\rm{e}^v=\widetilde{\nabla}_{\rm{e}^v}\rm{e}^c=\frac{1}{2}a^v, \nonumber\\
  && \widetilde{\nabla}_{\rm{e}^c}x^c=(-\frac{1}{2}f(x)-g(x,a)\rm{e})^c, \ \ \ \ \ \widetilde{\nabla}_{x^c}\rm{e}^c=-\frac{1}{2}f(x)^c,\nonumber\\
  && \widetilde{\nabla}_{x^v}\rm{e}^v=\widetilde{\nabla}_{\rm{e}^v}x^v=-\frac{1}{2}(f(x)+g(x,a)\rm{e})^c, \ \ \ \ \ \ \widetilde{\nabla}_{\rm{e}^v}x^c=(-\frac{1}{2}f(x)-g(x,a)\rm{e})^v,\\
  && \widetilde{\nabla}_{x^c}\rm{e}^v=\widetilde{\nabla}_{\rm{e}^c}x^v=\frac{1}{2}f(x)^v, \ \ \ \ \ \ \widetilde{\nabla}_{x^v}\rm{e}^c=(\frac{1}{2}f(x)+g(a,x)\rm{e})^v,\nonumber\\
  && \widetilde{\nabla}_{x^c}y^c=\frac{1}{2}g(f(x),y)\rm{e}^c, \ \ \ \ \ \ \widetilde{\nabla}_{x^c}y^v=\widetilde{\nabla}_{x^v}y^c=\frac{1}{2}g(f(x),y)\rm{e}^v,\ \ \ \ \ \ \widetilde{\nabla}_{x^v}y^v=0.\nonumber
\end{eqnarray}
\end{lem}
\begin{proof}

Easily we can see
\begin{eqnarray}
ad^*_ee=-a, \ \ \ \ ad^*_xe=g(a,x)e+f(x), \ \ \ \ ad^*_ex=ad^*_xy=0.
\end{eqnarray}
Now using lemma \ref{tangent bundle connection} completes the proof.
\end{proof}

\begin{lem}\label{lemma 5}
With the assumptions of Lemma \ref{lemma 4} for the curvature tensor $\widetilde{R}$ of the Riemannian manifold $(TG,\widetilde{g})$ we have:
\begin{eqnarray}
  && \widetilde{R}(x^c,y^c)z^c=(-\frac{1}{4}g(f(y),z)f(x)+\frac{1}{4}g(f(x),z)f(y)+\frac{1}{2}f(z)+g(z,a)\rm{e})^c, \nonumber\\
  && \widetilde{R}(x^c,y^c)z^v=(-\frac{1}{4}g(f(y),z)f(x)+\frac{1}{4}g(f(x),z)f(y)-\frac{1}{2}g(f(x),y)f(z))^v, \nonumber\\
  && \widetilde{R}(x^c,y^v)y^v=\frac{1}{4}g(f(x),y)f(x)^v+\frac{1}{2}g(f(x),y)(f(y)+g(y,a)\rm{e})^c, \nonumber\\
  && \widetilde{R}(x^v,y^v)y^v=0, \\
  && \widetilde{R}(x^c,{\rm{e}}^v){\rm{e}}^v=(\frac{1}{2}g(f(x),a){\rm{e}}-\frac{1}{4}(f^2(x)+g(f(x),a){\rm{e}})-g(a,x)a)^c, \nonumber\\
  && \widetilde{R}(x^v,{\rm{e}}^c){\rm{e}}^c=(\frac{1}{2}g(f(x),a){\rm{e}}-\frac{1}{4}f^2(x)-g(a,x)a)^v, \nonumber\\
  && \widetilde{R}(x^v,{\rm{e}}^v){\rm{e}}^v=(\frac{1}{2}g(f(x),a){\rm{e}}-\frac{1}{4}f^2(x)-\frac{1}{2}g(f(x),a){\rm{e}}+\frac{1}{4}g(x,a)a)^v.\nonumber \nonumber
\end{eqnarray}
\end{lem}
\begin{proof}
It is an immediate consequence of previous lemma.
\end{proof}

The following theorem shows that, similar to $\mathcal{G}_1$, the tangent bundle of Lie groups in $\mathcal{G}_2$ admits positive, negative and zero sectional curvatures.

\begin{theorem}
Consider the same assumptions of lemma \ref{lemma 4}. Then for any $G\in \mathcal{G}_1$, $(TG,\widetilde{g})$ at any point, admits positive, negative and zero sectional curvatures.
\end{theorem}

\begin{proof}
Lemma \ref{lemma 5} and some computations show that for $x,y \in \Gamma$ we have,
\begin{eqnarray}
&&\widetilde{K}(x^c,y^c) = K(x,y)=-\frac{3}{4}g(f(x),y)^2, \ \ \ \  \widetilde{K}(x^c,y^v) =-\frac{1}{2}g(f(x),y)^2,\\
&&\widetilde{K}(x^c,\rm{e}^c) = K(x,\rm{e})=-\frac{1}{4}\|f(x)\|^2, \ \ \ \ \widetilde{K}(x^c,\rm{e}^v) = \widetilde{K}(x^v,\rm{e}^c) = -\frac{1}{4}g(f^2(x),x)-g(a,x)^2, \nonumber
\end{eqnarray}
which obviously are negative. Similarly we can see $\widetilde{K}(x^v,\rm{e}^v) =\frac{1}{4}\|f(x)\|^2+\frac{1}{4}g(x,a)^2$ that is positive and $\widetilde{K}(x^v,y^v) = 0$, so the proof is completed.
\end{proof}
Although the previous theorem showed that $(TG,\widetilde{g})$ at any point, admits positive, negative and zero sectional curvatures but the following theorem proves that $(TG,\widetilde{g})$ is of non-positive Ricci curvature.

\begin{theorem}
Let $G\in \mathcal{G}_1$, then $(TG,\widetilde{g})$ is of non-positive Ricci curvature.
\end{theorem}

\begin{proof}
Suppose that $\frak {g}$ is the Lie algebra of $G$ and $\{u_1,...,u_n\}$ is an arbitrary orthonormal basis of $\Gamma$ such that $\{u_1,...,u_n,\rm{e}\}$ is an orthonormal basis for $\frak{g}$. Let $\widetilde{r}(\widetilde{x})=\widetilde{R}(\widetilde{x},\widetilde{x})$ be the Ricci curvature in direction $\widetilde{x}$. Then lemma \ref{lemma 5} together with some computations show that:
\begin{eqnarray}
&&\widetilde{r}({\rm{e}}^c)=-g(a,a)-\frac{1}{2}\sum_{i=1}^n g(f^2(u_i),u_i)-2\sum_{i=1}^ng(a,u_i)^2, \nonumber \\
&&\widetilde{r}({\rm{e}}^v)=\frac{-1}{4}g(a,a)-\frac{1}{2} \sum_{i=1}^n g(f^2(u_i),u_i)-\sum_{i=1}^ng(a,u_i)^2, \\
&&\widetilde{r}({x}^c)=\frac{-1}{2}g(f^2(x),x)-\frac{3}{2}\sum_{i=1}^n g(f^2(u_i),x)^2-2g(a,x)^2, \nonumber \\
&&\widetilde{r}({x}^v)=\frac{-1}{4}g(f^2(x),x)-\frac{3}{4}\sum_{i=1}^n g(f^2(u_i),x)^2. \nonumber
\end{eqnarray}
So $G$ is of non-positive Ricci curvature.
\end{proof}

{\large{\textbf{Acknowledgment.}}} We are grateful to the office of Graduate Studies of the University of Isfahan for their support. This research was supported by the Center of Excellence for
Mathematics at the University of Isfahan.
%%-------------------- BIBLIOGRAPHY------------------------

\bibliographystyle{amsplain}

\end{document}